\documentclass[a4paper,11pt]{amsart}

\usepackage[utf8]{inputenc}
\usepackage[all,cmtip]{xy}
\usepackage[english]{babel}
\usepackage{mathtools}
\usepackage{graphicx}
\usepackage{multicol}

\usepackage{amsmath,amssymb,amsthm,amsrefs,mathrsfs}
\usepackage{xfrac}
\usepackage{pifont}

\usepackage{array,tabularx}
\usepackage{xcolor}

\usepackage{tikz,tikz-cd}
\usetikzlibrary{positioning, shapes.geometric}
\usetikzlibrary{arrows,arrows.meta}
\usetikzlibrary{calc}
\tikzcdset{arrow style=tikz, diagrams={>={Computer Modern Rightarrow[width=5.5pt,length=3pt]}}}

\usepackage{enumitem}
\usepackage{cleveref}

\usepackage[retainorgcmds]{IEEEtrantools}

\theoremstyle{plain}
\newtheorem{thm}{Theorem}[section] 
\newtheorem*{thm*}{Theorem}
\newtheorem*{mainthm}{Main Theorem}
\newtheorem{prop}[thm]{Proposition}
\newtheorem{cor}[thm]{Corollary}
\newtheorem{lem}[thm]{Lemma}

\newtheorem*{conj}{Conjecture}

\theoremstyle{definition}
\newtheorem{defn}[thm]{Definition}
\newtheorem{exa}[thm]{Example}
\newtheorem{rem}[thm]{Remark}

\newcommand*{\myproofname}{Proof of Theorem \ref{1thm:inertness}}
\newenvironment{thmproof}[1][\myproofname]{\begin{proof}[#1]}{\end{proof}}

\newcommand*{\myproofnames}{Proof of Proposition \ref{prop:pdhmlgy}}

\newcommand{\N}{\mathbb{N}}
\newcommand{\Z}{\mathbb{Z}}
\newcommand{\Q}{\mathbb{Q}}

\newcommand{\calL}{\mathcal{L}}

\usepackage{xparse}
\makeatletter
\RenewDocumentCommand{\title}{om}{%
  \IfNoValueTF{#1}
     {\gdef\shorttitle{Connected Sums and the Vigu\'e-Poirrier Conjecture}}
     {\gdef\shorttitle{#1}}%
  \gdef\@title{#2}%
}
\makeatother

\usepackage{libertine}
\usepackage[libertine]{newtxmath}
\renewcommand{\mathbb}{\varmathbb}

\title{Loop Space Decompositions of Connected Sums and Applications to the Vigu\'e-Poirrier Conjecture}
\author{Sebastian Chenery}
\address{Mathematical Sciences, University of Southampton, Southampton SO17 1BJ, United Kingdom}
\email{s.chenery@soton.ac.uk}


\begin{document}

\maketitle

\begin{abstract}
Recent work of Beben and Theriault on decomposing based loop spaces of highly connected Poincar\'e Duality complexes has yielded new methods for analysing the homotopy theory of manifolds. In this paper we will expand upon these methods, which we will then apply to give new examples supporting a long standing question of rational homotopy theory: the Vigu\'e-Poirrier Conjecture.
\end{abstract}

\section{Introduction}

Let \(X\) be a simply connected space and let \(\calL X=Map(S^1,X)\) be its \textit{free loop space}. That is to say, \(\calL X\) is the space of continuous (not necessarily pointed) maps from \(S^1\) to \(X\). We write \(\Omega X\) for the subspace of pointed maps, called the \textit{based loop space} of \(X\). Such a space \(X\) is called \textit{rationally elliptic} if \(dim(\pi_*(X)\otimes\Q)<\infty\), and called\textit{ rationally hyperbolic} otherwise \cite{fht}. This paper uses techniques of homotopy theory to give new evidence supporting the Vigu\'e-Poirrier Conjecture, which concerns the rational hyperbolicity of \(X\). In particular, we will focus on the situation when \(\Omega X\) has the homotopy type of the based loop space of a connected sum of Poincar\'e Duality complexes, in the sense of \cite{wall1967poincare}. The Conjecture, stated below, was first given in \cite{vp}. 

\begin{conj}[Vigu\'e-Poirrier]
If \(X\) is rationally hyperbolic, then \(H_*(\calL X;\Q)\) grows exponentially.
\end{conj}

It was shown in the same paper that the Conjecture holds when \(X\) is a finite wedge sum of at least two spheres, i.e.~that such an \(X\) is an example of a space that is rationally hyperbolic and has the property that \(H_*(\calL X;\Q)\) grows exponentially. The Vigu\'e-Poirrier Conjecture may also be viewed as a development of another conjecture, due to Gromov \cite{gromov}, that when \(X\) is a closed manifold then \(H_*(\calL X;\Q)\) `almost always' grows exponentially. This has profound implications in Riemannian geometry, in which one may give a lower bound for the number of geometrically distinct closed geodesics on a simply connected closed Riemannian manifold \(M\) using the rate of growth in the dimension of \(H_*(\calL M;\Q)\). Indeed, if the Betti numbers of \(\calL X\) are unbounded, then \(X\) has infinitely many distinct closed geodesics in any Riemannian metric \cite{gromollmeyer}*{Theorem 4}.

The Vigu\'e-Poirrier Conjecture has also been shown to hold for several other classes of spaces. Notably for us, Lambrechts proved that the Conjectures holds for non-trivial connected sums of closed manifolds which are not monogenic in cohomology \cite{lambrechtsbetti}. This result in particular was strenghtened to the case when only one of the connect-summands was required to not be monogenic \cite{fhtloopii}*{Theorem 1.4}, and these methods can be expanded to include Poincar\'e Duality complexes which satisfy these assumptions. 

This brings us to our application, which broadens the class of spaces for which the Conjecture holds to Poincar\'e Duality complexes with the loop space homotopy type of certain connected sums; in particular, such a homotopy equivalence need not hold before taking loop spaces. The statement of the Main Theorem below uses the notions of \textit{inert} maps and \textit{good exponential growth} -- these are defined precisely in Definitions \ref{def:inert} and \ref{def:geg}, respectively. 

\begin{mainthm}\label{mainthm:vp}
    Let $n>3$ and let $M$ be a $n$-dimensional Poincar\'e Duality complex, and suppose that there exist \(n\)-dimensional Poincar\'e Duality complexes \(N\) and \(P\) such that \(\Omega M\simeq\Omega (N\#P)\). If \(P\) is rationally elliptic and the attaching map of the top-cell of \(P\) is inert, and if \(\calL(N\#P)\) has good exponential growth, then \(M\) satisfies the Vigu\'e-Poirrier Conjecture. \qed
\end{mainthm}

Good exponential growth of \(\calL(N\#P)\) occurs (for example) in cases when \(N\) is rationally hyperbolic, or when the cohomological condition of Lambrechts is satisfied for \(N\#P\). 

The Main Theorem arises from a homotopy theoretic analysis of spaces that have the homotopy type of connected sums of Poincar\'e Duality after taking based loops, but not before. It therefore shows that the Conjecture holds not just for connected sums which satisfy Lambrechts' condition, but also for spaces which share their based loop homotopy type. Thus we broaden the context for which the Vigu\'e-Poirrier Conjecture is known to hold. Such spaces exist in abundance: we produce an infinite family of examples in Section \ref{sec:conex} of this paper. In this way we continue a trend of recent work (see, for example, \cite{huangtheriault}) wherein results of rational homotopy theory are reformulated and generalised using \textit{integral} methods. 

The structure of this paper is as follows. Section \ref{sec:ther} establishes some modifications to a construction of Theriault \cite{t20}*{Section 8} in order to prove Theorem \ref{1thm:inertness}, which underpins much of what follows and provides a basis for how we will use inert maps. This is employed in Section \ref{sec:inert} to give an important fact in Lemma \ref{lem:inertness}, and then to give further results in the context of the homotopy theory of Poincar\'e Duality complexes, including a new proof of \cite{t20}*{Theorem 9.1(b)-(c)}. The methods we use for this analysis have been the subject of intense interest in recent times, and have been given further study in works such as \cite{KishimotoMinowa_BT}.

A key theorem of this paper is Theorem \ref{thm:PDconn}, which gives a general framework for our analysis. We follow this with Section \ref{sec:conex}, in which we give a family of examples of Poincar\'e Duality complexes which have the homotopy type of a connected sum after looping, but not before. We then conclude with Section \ref{sec:vp} by using Theorem \ref{thm:PDconn} to expand the class of spaces for which the Vigu\'e-Poirrier Conjecture is known to hold, which we give in Corollary \ref{cor:vp} and Theorem \ref{thm:vp}.


\section{Adapting a Construction of Theriault}\label{sec:ther}

We start by making some adjustments to a construction of Theriault from \cite{t20}*{Section 8} in order to prove a more general result, which forms the basis of the theory that follows. To begin, we establish some notation. Unless otherwise stated, all spaces are assumed to be simply connected. Moreover, given a wedge of spaces \(\bigvee_{i=1}^n X_i\), let \(p_j:\bigvee_{i=1}^n X_i\rightarrow X_j\) denote the pinch map to the \(j^{th}\) summand (note that the inclusion of the \(j^{th}\) summand is a right homotopy inverse for \(p_j\)). We also give the following definition.

\begin{defn}\label{def:inert}
    For a homotopy cofibration \(A\xrightarrow{f}B\xrightarrow{j}C\) the map \(f\) is called \textit{inert} if \(\Omega j\) has a right homotopy inverse.
\end{defn} 

Our focus in this section will be on homotopy cofibrations of the form \[\Sigma A \xrightarrow{f} X\vee Y\xrightarrow{q} C\] and the key to our considerations will be the following homotopy commutative diagram of homotopy cofibrations \begin{equation}\label{1dgm:setup1}
    \begin{tikzcd}[row sep=1.5em, column sep = 1.5em]
        && Y \arrow[rr, equal] \arrow[dd, hook] && Y \arrow[dd] \\
        &&&&& \\
        \Sigma A \arrow[rr, "f"] \arrow[dd, equal] && X\vee Y  \arrow[rr, "q"] \arrow[dd, "p_1"] && C \arrow[dd, "\varphi"] \\
        &&&&& \\
        \Sigma A \arrow[rr, "p_1\circ f"] && X \arrow[rr, "j"] && M
    \end{tikzcd}
\end{equation}
where the bottom-right square is a homotopy pushout. Our goal is to prove the following theorem. 

\begin{thm}\label{1thm:inertness}
Consider Diagram (\ref{1dgm:setup1}). If the map \(\Omega j\) has a right homotopy inverse, then so do \(\Omega \varphi\) and \(\Omega q\). In particular, if the composite \(p_1\circ f\) is inert, then so is \(f\).
\end{thm}


This result is a refinement of \cite{t20}*{Lemma 8.1}, albeit in a slightly different formulation - indeed, our method of proof is similar.  First, recall that for two path connected and based spaces $X$ and $Y$, the \textit{(left) half-smash} of $X$ and $Y$ is the quotient space \[X\ltimes Y=(X\times Y)/(X\times y_0)\] where $y_0$ denotes the basepoint of $Y$. We begin by stating one of the main theorems of \cite{t20}.

\begin{thm}[Theriault] \label{1thm:data}
Suppose there exists a homotopy commutative diagram
\[
    \begin{tikzcd}[row sep=1.5em, column sep = 1.5em]
        && E \arrow[rr, "\alpha"] \arrow[dd] && E' \arrow[dd] \\
        &&&&& \\
        \Sigma A \arrow[rr, "f"] && B \arrow[rr] \arrow[dd, "h"] && C \arrow[dd, "h'"] \\
        &&&&& \\
        && Z \arrow[rr, equal] && Z
    \end{tikzcd}
\]
where the middle and right columns are homotopy fibrations, the map \(\alpha\) is an induced map of fibres and the middle row is a homotopy cofibration. If $\Omega h$ has a right homotopy inverse, then there exists a homotopy cofibration \[\Omega Z\ltimes \Sigma A\xrightarrow{\theta}E\rightarrow E'\] for some map \(\theta\). \hfill $\square$
\end{thm}

Note the special case in which \(C=Z\) and \(h'\) is the identiy map, which implies that \(E'\) is contractible and therefore that \(\theta\) is a homotopy equivalence. This gives the following corollary, a version of which the reader will also find in \cite{bt2}*{Proposition 3.5}, where it first appeared. Note that the need for the suspension \(\Sigma A\) is dropped.

\begin{cor}[Beben-Theriault]\label{1cor:splitfib} 
Suppose there is a homotopy cofibration \[A\xrightarrow{f}B\xrightarrow{h}C\] such that the map $\Omega h$ has a right homotopy inverse. Then there exists a homotopy fibration \[\Omega C\ltimes A\rightarrow B\xrightarrow{h}C.\] Moreover, this homotopy fibration splits after looping, so there is a homotopy equivalence $\Omega B\simeq \Omega C\times\Omega(\Omega C\ltimes A).$ \hfill $\square$
\end{cor}

Returning to the situation of Diagram (\ref{1dgm:setup1}), since $p_1\circ f$ is inert, the map \(\Omega j\) has a right homotopy inverse. Let $E''$ be the homotopy fibre of \(j\). Corollary \ref{1cor:splitfib} applies to the homotopy cofibration in the bottom row of (\ref{1dgm:setup1}), which implies that there is a homotopy equivalence \(E''\simeq\Omega M\ltimes \Sigma A\). Equivalently, there is a homotopy cofibration 
\begin{equation}\label{eqn:theta}
\Omega M\ltimes\Sigma A\xrightarrow{\theta} E''\rightarrow \ast.
\end{equation} 
Now, let \(s\) denote a right homotopy inverse of \(\Omega j\), and \(t\) that of \(\Omega p_1\). Then the composite \(\Omega q\circ t\circ s\) is a right homotopy inverse for \(\Omega \varphi\). Let \(h=j\circ p_1\) and let \(E\) and \(E'\) denote the homotopy fibres of \(h\) and \(\varphi\), respectively. We have the following homotopy commutative diagram
\begin{equation}\label{1dgm:data2} 
    \begin{tikzcd}[row sep=1.5em, column sep = 1.5em]
        && E \arrow[rr, "\alpha"] \arrow[dd] && E' \arrow[dd] \\
        &&&&& \\
        \Sigma A \arrow[rr, "f"] && X\vee Y \arrow[rr, "q"] \arrow[dd, "h"] && C \arrow[dd, "\varphi"] \\
        &&&&& \\
        && M \arrow[rr, equal] && M
    \end{tikzcd}
\end{equation}
where the middle and right columns are homotopy fibrations, the map \(\alpha\) is an induced map of fibres and the middle row is a homotopy cofibration. Therefore, by Theorem \ref{1thm:data}, there exists a homotopy cofibration (here we write \(\theta_f\) to distinguish from the map \(\theta\) in (\ref{eqn:theta}))
\begin{equation}\label{1cofib1}    
    \Omega M\ltimes\Sigma A\xrightarrow{\theta_f} E\xrightarrow{\alpha} E'. 
\end{equation}
Moreover, the right homotopy inverse for \(\Omega \varphi\) enables us to apply Corollary \ref{1cor:splitfib} to the right-most column of (\ref{1dgm:setup1}), so we have homotopy equivalences \[E'\simeq\Omega M\ltimes Y\text{\; and \;}\Omega C\simeq\Omega M\times\Omega(\Omega M\ltimes Y).\] The proof strategy for Theorem \ref{1thm:inertness} will be as follows: we wish to gain more control over the homotopy class of the first equivalence, and use this knowledge to deduce further facts about the second. 

The next step is to consider the homotopy fibration diagram
\begin{equation}\label{1dgm:fib1}
    \begin{tikzcd}[row sep=1.5em, column sep = 1.5em]
        E \arrow[rr, dashed, "l"] \arrow[dd] && E'' \arrow[dd] \\
        &&& \\
        X\vee Y \arrow[rr, "p_1"] \arrow[dd, "h"] && X \arrow[dd, "j"] \\
        &&& \\
        M \arrow[rr, equal] && M
    \end{tikzcd}
\end{equation}
in which the bottom square is commutative by definition of the map \(h\), so the induced map of fibres \(l\) exists. In \cite{t20}*{Remark 2.7}, a naturality condition is given for Theorem \ref{1thm:data}, which is satisfied by virtue of (\ref{1dgm:fib1}). Thus there is a commutative diagram of homotopy cofibrations which combines (\ref{eqn:theta}) and (\ref{1cofib1}):
\begin{equation}\label{1dgm:cofib1}
    \begin{tikzcd}[row sep=1.5em, column sep = 1.5em]
        \Omega M \ltimes \Sigma A \arrow[rr, "\theta_f"] \arrow[dd, equal] && E  \arrow[rr, "\alpha"] \arrow[dd, "l"] && E' \arrow[dd] \\
        &&&&& \\
        \Omega M \ltimes \Sigma A \arrow[rr, "\theta"] && E'' \arrow[rr] && *.
    \end{tikzcd}
\end{equation}
Note that since \(\theta\) is a homotopy equivalence, (\ref{1dgm:cofib1}) implies that the map \(\theta_f\) always has a left homotopy inverse. Moreover, observe also that in constructing the above we only considered the behaviour of $f$ when restricted to $X$. We record this in the lemma below, for ease of reference.

\begin{lem}\label{1lem:thetaleftinv}
With the set-up of Diagram (\ref{1dgm:cofib1}) above, the map \(\theta_f\) has a left homotopy inverse whose homotopy class depends only on the homotopy class of the composite \(f\) when restricted to \(X\). \hfill $\square$
\end{lem}

This enables us to switch focus for the time being, and consider the special case in which \(C\simeq M\vee Y\). Diagram (\ref{1dgm:setup1}) now becomes the homotopy cofibration diagram
\begin{equation}\label{1dgm:cofib2}
    \begin{tikzcd}[row sep=1.5em, column sep = 1.5em]
        && Y \arrow[rr, equal] \arrow[dd, hook] && Y \arrow[dd, hook] \\ 
        &&&&& \\
        \Sigma A \arrow[rr, "f"] \arrow[dd, equal] && X\vee Y  \arrow[rr, "j\vee 1"] \arrow[dd, "p_1"] && M\vee Y \arrow[dd, "p_1"] \\
        &&&&& \\
        \Sigma A \arrow[rr, "p_1\circ f"] && X \arrow[rr, "j"] && M
    \end{tikzcd}
\end{equation}
where we have \(q=j\vee1\) and \(\varphi=p_1\). Since the map \(\Omega p_1\) has a right homotopy inverse, we may apply Corollary \ref{1cor:splitfib} to the homotopy cofibration in the right-most column of (\ref{1dgm:cofib2}), again giving a homotopy equivalence \(E'\simeq \Omega M\ltimes Y\). Thus Diagram (\ref{1dgm:data2}) becomes
\begin{equation}\label{1dgm:fib1'}
    \begin{tikzcd}[row sep=1.5em, column sep = 1.5em]
        && E \arrow[rr, "\alpha'"] \arrow[dd] && \Omega M\ltimes Y \arrow[dd] \\
        &&&&& \\
        \Sigma A \arrow[rr, "f"] && X\vee Y \arrow[rr, "j\vee 1"] \arrow[dd, "h"] && M\vee Y \arrow[dd, "p_1"] \\
        &&&&& \\
        && M \arrow[rr, equal] && M
    \end{tikzcd}
\end{equation}
and, analogously to (\ref{1cofib1}), there is a homotopy cofibration
\begin{equation*}
    \Omega M\ltimes\Sigma A\xrightarrow{\theta_f'} E\xrightarrow{\alpha'} \Omega M\ltimes Y.
\end{equation*}
Noting that the upper square of (\ref{1dgm:fib1'}) is a homotopy pullback and arguing exactly as in the proof of \cite{t20}*{Lemma 8.3} gives the following.

\begin{lem}\label{1lem:alpha'rightinv}
The map \(\alpha'\) has a right homotopy inverse \(r:\Omega M\ltimes Y\rightarrow E\) such that the composite \(l\circ r\) is null homotopic. \hfill $\square$
\end{lem} 

Combining this with the general situation, we have homotopy cofibrations \[\Omega M\ltimes\Sigma A\xrightarrow{\theta_f} E\xrightarrow{\alpha} E'\text{\; and \;}\Omega M\ltimes\Sigma A\xrightarrow{\theta_f'} E\xrightarrow{\alpha'} \Omega M\ltimes Y.\] By Lemma \ref{1lem:thetaleftinv} there is a left homotopy inverse \(k\) for both $\theta_f$ and $\theta_f'$. Lemma \ref{1lem:alpha'rightinv} gives a right homotopy inverse $r$ for $\alpha'$, and since \(l\circ r\simeq \ast\), Diagram (\ref{1dgm:cofib1}) implies that \(k\circ r\simeq \ast\). By \cite{t20}*{Lemma 8.5}, this implies that the composite 
\begin{equation}\label{1equiv1}
    \Omega M \ltimes Y\xrightarrow{r}E\xrightarrow{\alpha} E'
\end{equation} 
is a homotopy equivalence. This achieves our first goal of gaining more control over the homotopy equivalence \(E'\simeq\Omega M\ltimes Y\); we will use the fact that it factors through \(E\) to prove Theorem \ref{1thm:inertness}. Recall that applying Corollary \ref{1cor:splitfib} to the right-most column of Diagram (\ref{1dgm:data2}) yields a homotopy equivalence \[\Omega C\simeq \Omega M\times \Omega(\Omega M\ltimes Y).\] 

\begin{thmproof}
We have already shown that the map \(\Omega \varphi\) has a right homotopy inverse given by \(\Omega q\circ s\circ t\), due to homotopy commutativity of (\ref{1dgm:setup1}), thus it remains to prove that the map \(\Omega q\) also has a right homotopy inverse. We shall do this by showing that the above homotopy equivalence for \(\Omega C\) factors through \(\Omega q\), from which the existence of a right homotopy inverse for \(\Omega q\) follows immediately. Let \(\lambda=s\circ t\), which is a right homotopy inverse for the map \(\Omega h\). Taking loops on the middle and right columns of (\ref{1dgm:data2}) gives
\begin{equation}\label{1dgm:loopsetup}
    \begin{tikzcd}[row sep=1.5em, column sep = 1.5em]
        \Omega E \arrow[rr, "\Omega \alpha"] \arrow[dd] && \Omega E' \arrow[dd] \\
        &&& \\
        \Omega(X\vee Y) \arrow[rr, "\Omega q"] \arrow[dd, "\Omega h"] && \Omega C \arrow[dd, "\Omega \varphi"] \\
        &&& \\
        \Omega M \arrow[rr, equal] && \Omega M.
    \end{tikzcd}
\end{equation}
Letting \(r'\) denote \(\Omega M\ltimes Y\xrightarrow{r}E\rightarrow X\vee Y\), we have the composite \[e:\Omega M\times \Omega(\Omega M\ltimes Y)\xrightarrow{\lambda\times\Omega r'}\Omega(X\vee Y)\times\Omega(X\vee Y)\xrightarrow{\mu}\Omega(X\vee Y)\xrightarrow{\Omega q}\Omega C\] where \(\mu\) is the loop multiplication. The homotopy commutativity of (\ref{1dgm:loopsetup}) together with the fact that \(\Omega q\) is an \(H\)-map implies that that \(e\) is a homotopy equivalence, so the proof is complete.
\end{thmproof}

\section{Inert Maps and Loop Space Decompositions}\label{sec:inert}

We wish to apply Theorem \ref{1thm:inertness} to the situation of connected sums: suppose we have two homotopy cofibrations of simply connected spaces \[\Sigma A\xrightarrow{f} B\xrightarrow{j} C\text{\; and \;}\Sigma A\xrightarrow{g} D\xrightarrow{l} E.\] Consider the composite $f\check{+}g:\Sigma A\xrightarrow{\sigma}\Sigma A\vee \Sigma A \xrightarrow{f\vee g}B\vee D$, where $\sigma$ denotes the suspension co-multiplication on $\Sigma A$ (we use the symbol \(\check{+}\) to distinguish from the case in which \(B=D\), where using \(+\) would introduce an ambiguity). The homotopy cofibre of $f\check{+}g$ is called the \textit{generalised connected sum} of $C$ and $E$ over $\Sigma A$, written $C\#_{\Sigma A}E$. When the space $A$ is clear, we will often omit the subscript. 

Note that $p_1\circ(f\check{+}g)\simeq f$. We have the diagram below, in which each complete row and column is a homotopy cofibration, and the bottom-right square is a homotopy pushout. We label the induced map $C\#_{\Sigma A}E\rightarrow C$ by $h$.
\begin{equation}\label{dgm:setup}
    \begin{tikzcd}[row sep=1.5em, column sep = 1.5em]
        && D \arrow[rr, equal] \arrow[dd, hook] && D \arrow[dd] \\
        &&&&& \\
        \Sigma A \arrow[rr, "f\check{+}g"] \arrow[dd, equal] && B\vee D  \arrow[rr, "q"] \arrow[dd, "p_1"] && C\#_{\Sigma A}E \arrow[dd, "h"] \\
        &&&&& \\
        \Sigma A \arrow[rr, "f"] && B \arrow[rr, "j"] && C.
    \end{tikzcd}
\end{equation}
The following lemma is a stronger version of \cite{chenfib}*{Lemma 2.2}; it follows immediately from Theorem \ref{1thm:inertness} and by applying Corollary \ref{1cor:splitfib} to the rightmost column of Diagram (\ref{dgm:setup}).

\begin{lem} \label{lem:inertness}
Take the setup of Diagram (\ref{dgm:setup}). If the map \(f\) is inert, then so is \(f\check{+}g\). Moreover, there is a homotopy equivalence \(\Omega (C\#_{\Sigma A}E)\simeq\Omega C\times\Omega(\Omega C\ltimes D)\). \hfill \(\square\)
\end{lem}

\begin{rem}
Of particular note is the fact that Lemma \ref{lem:inertness} implies that the map \(f\check{+}g\) inherits inertness from \(f\), regardless of the homotopy class of the map \(g\).
\end{rem}

We will use the above framework in our consideration of connected sums of Poincar\'e Duality complexes. For such a complex, there exists a cell structure that has a single top-dimensional cell, and we may define the connected sum operation similarly to that of manifolds. Namely, for two $n$-dimensional Poincar\'e Duality complexes $M$ and $N$, the space $M\#N$ is formed by removing an $n$-dimensional open disc from the interior of the top-cells of $M$ and $N$ and joining the resulting complexes along their boundaries. Up to homotopy, $M\# N$ coincides with the generalised connected sum $M\#_{S^{n-1}}N$ \cite{wall1967poincare}. This gives rise to the following example.

\begin{exa}
    If \(M\) and \(N\) are smooth, closed, oriented manifolds manifolds then we choose \(M\#_{S^{n-1}}N\) such that it coincides (up to homotopy) with the usual orientation preserving connected sum of \(M\) and \(N\).
\end{exa}

In pursuit of our homotopy theoretic analysis, we seek a framework whereby it may be shown that a Poincar\'e Duality complex has the homotopy type of a connected sum, after looping. To begin to give this we have the following proposition, which is a restatement of \cite{t20}*{Theorem 9.1 (b)-(c)}, though we provide a new proof.

\begin{prop}[Theriault]\label{prop:PDconnsum}
Let $M$ and $N$ be two Poincar\'e Duality complexes of dimension $n$, where $n>3$, such that the attaching map of the top-cell of $M$ is inert. Then there is a homotopy equivalence \[\Omega (M\#N)\simeq \Omega M\times\Omega(\Omega M\ltimes N_{n-1})\] where $N_{n-1}$ denotes the $(n-1)$-skeleton of $N$. Furthermore, the attaching map of the top-cell of $M\#N$ is inert.
\end{prop}

\begin{proof}
Let $f_M$ and $f_N$ be the attaching maps of the top-cells of $M$ and $N$, respectively, onto their $(n-1)$-skeleta $M_{n-1}$ and $N_{n-1}$. We have homotopy cofibrations \[S^{n-1}\xrightarrow{f_M}M_{n-1}\rightarrow M\text{\; and \;}S^{n-1}\xrightarrow{f_N}N_{n-1}\rightarrow N.\] Similar to Diagram (\ref{dgm:setup}), we have the following homotopy cofibration diagram.
\begin{equation*}
    \begin{tikzcd}[row sep=1.5em, column sep = 1.5em]
        && N_{n-1} \arrow[rr, equal] \arrow[dd, hook] && N_{n-1} \arrow[dd] \\
        &&&&& \\
        S^{n-1} \arrow[rr, "f_M\check{+}f_N"] \arrow[dd, equal] && M_{n-1}\vee N_{n-1} \arrow[rr] \arrow[dd, "p_1"] && M\#N \arrow[dd, "h"] \\
        &&&&& \\
        S^{n-1} \arrow[rr, "f_M"] && M_{n-1}  \arrow[rr] && M.
    \end{tikzcd} 
\end{equation*}
The result then follows from Lemma \ref{lem:inertness}.
\end{proof}

So far we have only used Theorem \ref{1thm:inertness} to show that a sum of attaching maps is inert. The next theorem gives conditions for an attaching map to be inert without first supposing that it is formed by the \(\check{+}\) operation. 

\begin{thm} \label{thm:PDconn}
Let $n>3$ and suppose that $M$, $N$ and $P$ are $n$-dimensional Poincar\'e Duality complexes. Let \(f_M\) and \(f_P\) denote the attaching maps of the top-cells of $M$ and $P$, respectively, and suppose further that: 
\begin{enumerate}
    \item[(i)] $M_{n-1}\simeq P_{n-1}\vee N_{n-1}$;
    \item[(ii)] the composite $p_1\circ f_M:S^{n-1}\rightarrow P_{n-1}$ is inert;
    \item[(iii)] the homotopy cofibre of $p_1\circ f_M$, \(Q\), is such that \(\Omega Q\simeq\Omega P\);
    \item[(iv)] the map \(f_P\) is inert.
\end{enumerate}
Then the attaching map $f_M$ is inert and $\Omega M\simeq \Omega(N\#P)$.
\end{thm}

\begin{proof}
From \textit{(i)} we have the following homotopy cofibration diagram
\begin{equation}\label{dgm:PDconn}
    \begin{tikzcd}[row sep=3em, column sep = 3em]
        & N_{n-1} \arrow[r, equal] \arrow[d, hook] & N_{n-1}  \arrow[d] \\
        S^{n-1} \arrow[r, "f_M"] \arrow[d, equal] & P_{n-1}\vee N_{n-1} \arrow[r] \arrow[d, "p_1"] & M \arrow[d] \\
        S^{n-1} \arrow[r, "p_1\circ f_M"] & P_{n-1} \arrow[r] & Q.
    \end{tikzcd}
\end{equation}
Condition \textit{(ii)} places us in the situation of Theorem \ref{1thm:inertness}, therefore \(f_M\) is inert and there is a homotopy equivalence \(\Omega M\simeq\Omega Q\times \Omega(\Omega Q\ltimes N_{n-1})\). By condition \textit{(iii)}, we therefore have
\begin{equation}\label{eqn:generalthm}
    \Omega M\simeq\Omega P\times \Omega(\Omega P\ltimes N_{n-1}).
\end{equation}
Now consider the connected sum $N\#P$. There is a homotopy cofibration diagram
\begin{equation*}
    \begin{tikzcd}[row sep=3em, column sep = 3em]
        & N_{n-1} \arrow[r, equal] \arrow[d, hook] & N_{n-1}  \arrow[d] \\
        S^{n-1} \arrow[r, "f_N\check{+}f_P"] \arrow[d, equal] & P_{n-1}\vee N_{n-1} \arrow[r] \arrow[d, "p_1"] & N\#P \arrow[d] \\
        S^{n-1} \arrow[r, "f_P"] & P_{n-1} \arrow[r] & P.
    \end{tikzcd}
\end{equation*}
By \textit{(iv)}, Lemma \ref{lem:inertness} gives a homotopy equivalence \(\Omega(N\#P)\simeq\Omega P\times \Omega(\Omega P\ltimes N_{n-1})\) and consequently that \(\Omega M\simeq \Omega(N\#P)\), due to (\ref{eqn:generalthm}).
\end{proof}

\section{A Constructive Example}\label{sec:conex}

In order for Theorem \ref{thm:PDconn} to carry weight and not be vacuous, we should provide an example of a Poincar\'e Duality complex with the homotopy type of a connected sum after looping, but not before. Indeed, in this section we will give an explicit family of such examples. Letting \(n>2\) be an integer, we take \[w_1:S^{2n-1}\rightarrow S^n\vee S^n\text{\;, \;}w_2:S^{2n-1}\rightarrow S^{n-1}\vee S^{n+1}\text{\; and \;}w_3:S^{2n-2}\rightarrow S^n\vee S^{n-1}\] to be the Whitehead products attaching the top-dimensional cells of the sphere products \(S^n\times S^n\), \(S^{n-1}\times S^{n+1}\) and \(S^n\times  S^{n-1}\), respectively. Furthermore, let \(\eta\) denote the classical Hopf map and let \(r=2n-4\). We form three composites, namely 
\begin{gather*}
    f_1:S^{2n-1}\xrightarrow{w_1}S^n\vee S^n\xhookrightarrow{i_{1,2}} S^n\vee S^n\vee S^{n-1} \vee S^{n+1} \\
    f_2:S^{2n-1}\xrightarrow{w_2} S^{n-1}\vee S^{n+1}\xhookrightarrow{i_{3,4}} S^n\vee S^n\vee S^{n-1} \vee S^{n+1} \\
    f_3:S^{2n-1}\xrightarrow{\Sigma^r\eta}S^{2n-2}\xrightarrow{w_3}S^n\vee S^{n-1}\xhookrightarrow{i_{2,3}} S^n\vee S^n\vee S^{n-1} \vee S^{n+1}
\end{gather*}
where  \(i_{j,k}\) is the inclusion of the \(j^{th}\) and \(k^{th}\) sphere. Before continuing with our construction, we will verify that the composite \(w_3\circ\Sigma^r\eta\) is essential (i.e. not null homotopic). To see this, we will study adjoints. In what follows, we let \(E\) denote the usual suspension map \(X\rightarrow \Omega \Sigma X\), and for a map of spaces \(f:\Sigma X\rightarrow Y\) we denote its adjoint by \(\hat{f}:X\rightarrow \Omega Y\). Recall that such a map \(f\) is essential if and only if \(\hat{f}\) is, and that \(\hat{f}\) is homotopic to the composite \[X\xrightarrow{E}\Omega\Sigma X\xrightarrow{\Omega f}\Omega Y.\] 

\begin{lem}\label{lem:ess}
    The composite \(w_3\circ\Sigma^r\eta\) is essential.
\end{lem}

\begin{proof}
In what follows, let \(\gamma=w_3\circ\Sigma^r\eta\). We will describe its adjoint \(\hat{\gamma}\) to show it is essential. We first consider the following diagram
    \begin{equation}\label{dgm:adjoints}
        \begin{tikzcd}[row sep=3em, column sep = 3em]
            S^{2n-2} \arrow[r, "\Sigma^{r-1}\eta"] \arrow[d, "E"] & S^{2n-3} \arrow[r, "\widehat{w_3}"] \arrow[d, "E"] & \Omega(S^n\vee S^{n-1}) \arrow[d, equals] \\
            \Omega S^{2n-1} \arrow[r, "\Omega\Sigma^r\eta"] \arrow[rr, bend right=20, "\Omega\gamma"] & \Omega S^{2n-2} \arrow[r, "\Omega w_3"] & \Omega(S^n\vee S^{n-1})
        \end{tikzcd} 
    \end{equation}
in which right square commutes by definition of the adjoint, the left square homotopy commutes by the naturality of the suspension map \(E\), and the bottom row is the definition of \(\Omega\gamma\). Thus the whole diagram commutes up to homotopy, and in particular we have \(\hat{\gamma}\simeq\hat{w}\circ\Sigma^{r-1}\eta\). Now, by the Hilton-Milnor Theorem we have a homotopy equivalence \[\Omega(S^n\vee S^{n-1})\simeq \prod_{k,l\in\N} \Omega S^{(k+l)n+(1-k-2l)}\] in which \(k=l=1\) gives a factor of \(\Omega S^{2n-2}\) and the composite \[\Omega S^{2n-2}\hookrightarrow \prod_{k,l\in\N} \Omega S^{(k+l)n+(1-k-2l)}\xrightarrow{\simeq}\Omega(S^n\vee S^{n-1})\] is homotopic to \(\Omega w_3\). Letting \(p:\Omega(S^n\vee S^{n-1})\rightarrow \Omega S^{2n-2}\) be the projection map from the product to this factor, consider the diagram
    \begin{equation}\label{dgm:adjoints2}
        \begin{tikzcd}[row sep=3em, column sep = 3em]
            S^{2n-2} \arrow[d, "\Sigma^{r-1}\eta"] \arrow[r, "\hat{\gamma}"]& \Omega(S^n\vee S^{n-1}) \arrow[r, "p"] & \Omega S^{2n-2} \\
            S^{2n-3} \arrow[ur, "\widehat{w_3}"] \arrow[r, "E"] & \Omega S^{2n-2}. \arrow[u, "\Omega w_3"] \arrow[ur, equals]
        \end{tikzcd} 
    \end{equation}
In particular, the right-hand triangle commutes, so the whole diagram commutes up to homotopy. To check that \(\hat{\gamma}\) is essential, it therefore it suffices to check that we have \(\Sigma^{r-1}\eta\) is not null homotopic, which always holds since \(r>0\).
\end{proof}

We return now to the main content of this long-form example. Let \(\sigma\) be the usual comultiplication on \(S^{2n-1}\) and let \(\nabla\) denote the fold map; we write \[\sigma'=(1\vee\sigma)\circ\sigma\text{\; and \;}\nabla'=\nabla\circ(1\vee\nabla)\] for three-fold uses of these maps. Consider now the composite \[f:S^{2n-1}\xrightarrow{\sigma'}\bigvee^3S^{2n-1}\xrightarrow{f_1\vee f_2\vee f_3} \bigvee^3(S^n\vee S^n\vee S^{n-1} \vee S^{n+1})\xrightarrow{\nabla'}S^n\vee S^n\vee S^{n-1} \vee S^{n+1}\] Letting \(p_{j,k}\) be the pinch map to the \(j^{th}\) and \(k^{th}\) sphere in \(S^n\vee S^n\vee S^{n-1} \vee S^{n+1}\) note that we have
\begin{equation}\label{eg:fs}
    p_{1,2}\circ f\simeq w_1\text{\;, \;}p_{3,4}\circ f\simeq w_2\text{\; and \;}p_{2,3}\circ f\simeq\gamma.
\end{equation}
The complex we wish to study is the homotopy cofibre of the map \(f\), which we shall denote by \(\mathcal{M}_n\). The construction of the map \(f\) implies that \(\mathcal{M}_n\) does not have the homotopy type of a connected sum before looping; if this were true, \(f\) would be representable in the form \(f\simeq g\check{+}h\). This is not the case, by virtue of Lemma \ref{lem:ess}, otherwise we would have \(\mathcal{M}_n\simeq(S^n\times S^n)\#(S^{n-1}\times S^{n+1})\). After looping, however, we have the following homotopy equivalence.

\begin{lem}\label{lem:pdnonconsumloop}
    There is a homotopy equivalence \(\Omega \mathcal{M}_n\simeq\Omega((S^n\times S^n)\#(S^{n-1}\times S^{n+1}))\).
\end{lem}

\begin{proof}
In the context of Theorem \ref{thm:PDconn}, take \(N=S^n\times S^n\) and \(P=S^{n-1}\times S^{n+1}\). We have the following homotopy cofibration diagram
    \begin{equation}\label{dgm:nonconnsumexa}
        \begin{tikzcd}[row sep=1.5em, column sep = 1.5em]
            && S^n\vee S^n \arrow[rr, equal] \arrow[dd, hook] && S^n\vee S^n \arrow[dd] \\
            &&&&& \\
            S^{2n-1} \arrow[rr, "f"] \arrow[dd, equal] && S^n\vee S^n\vee S^{n-1} \vee S^{n+1} \arrow[rr] \arrow[dd, "p_{3,4}"] && \mathcal{M}_n \arrow[dd, "h"] \\            
            &&&&& \\
            S^{2n-1} \arrow[rr, "w_2"] && S^{n-1}\vee S^{n+1}  \arrow[rr] && S^{n-1}\times S^{n+1}
        \end{tikzcd} 
    \end{equation}    
from which, since the map \(w_2\) is inert, Theorem \ref{thm:PDconn} gives the homotopy equivalence.
\end{proof}

So, the complex \(\mathcal{M}_n\) is not a connected sum before looping, but does have the loop space homotopy type of one. Our construction also carries a deeper structure, as the following lemma shows.

\begin{lem}\label{lem:pdnonconsumpd}
    \(\mathcal{M}_n\) is a Poincar\'e Duality complex.
\end{lem}

\begin{proof}
    We will show that \(\mathcal{M}_n\) inherits a Poincar\'e Duality structure via an algebra isomorphism \(H^*(\mathcal{M}_n)\cong H^*((S^n\times S^n)\#(S^{n-1}\times S^{n+1}))\). To begin, consider taking the pinch map \(p_{1,2}\) or \(p_{2,3}\) in the middle-column of Diagram (\ref{dgm:nonconnsumexa}), instead of \(p_{3,4}\). Via (\ref{eg:fs}) we therefore have, in addition to the map \(h:\mathcal{M}_n\rightarrow S^{n-1}\times S^{n+1}\), maps \[h':\mathcal{M}_n\rightarrow S^n\times S^n\text{\; and \;}h'':\mathcal{M}_n\rightarrow C_\gamma\] where \(C_\gamma\) denotes the homotopy cofibre of \(\gamma\). It has cells in dimensions 0, \(n-1\), \(n\) and \(2n\), so all cup products in \(H^*(C_\gamma)\) are trivial. Next we must establish notation: we write \[H^*(S^n\times S^n)=\Z\langle x_1,x_2,\mu_1\rangle\text{\; and \;}H^*(S^{n-1}\times S^{n+1})=\Z\langle x_3,x_4,\mu_2\rangle\] where \(|x_1|=|x_2|=n\), \(|x_3|=n-1\), \(|x_4|=n+1\), \(|\mu_1|=|\mu_2|=2n\), and we have cup products \(x_1\cup x_2=\mu_1\) and \(x_3\cup x_4=\mu_2\). We shall do similarly for \(\mathcal{M}_n\): it has cells in dimensions 0, \(n-1\), \(n\), \(n+1\) and \(2n\), so may write \(H^*(\mathcal{M}_n)=\Z\langle y_1, y_2, y_3, y_4, \mu_{\mathcal{M}_n}\rangle\) where \(|y_1|=|y_2|=n\), \(|y_3|=n-1\), \(|y_4|=n+1\), \(|\mu_{\mathcal{M}_n}|=2n\). By construction, the induced homomorphisms \(h^*\) and \((h')^*\) on cohomology take generators to generators; explicitly, we have 
    \begin{gather*}
    h^*(x_1)=y_1\text{, }h^*(x_2)=y_2\text{, }(h')^*(x_3)=y_3\text{, }(h')^*(x_4)=y_4 
    \\ \text{ and }h^*(\mu_1)=(h')^*(\mu_2)=\mu_{\mathcal{M}_n}
    \end{gather*} 
    which in turn induces cup products \(y_1\cup y_2=y_3\cup y_4=\mu_{\mathcal{M}_n}\), which are the only possible cup products for dimensional reasons. Thus there is a clear algebra isomorphism \(H^*(\mathcal{M}_n)\cong H^*((S^n\times S^n)\#(S^{n-1}\times S^{n+1}))\) by virtue of which \(\mathcal{M}_n\) gains a Poincar\'e Duality structure. 
\end{proof}

Lemmas \ref{lem:pdnonconsumloop} and \ref{lem:pdnonconsumpd} combine to give the following Proposition, constituting the thrust of this section. 

\begin{prop}\label{prop:pdnonconsum}
    For each integer \(n>2\) there exists a \(2n\)-dimensional Poincar\'e Duality complex \(\mathcal{M}_n\) with the homotopy type of a connected sum after looping, but not before. \hfill \qed  
\end{prop}

\begin{rem}
    Note that there was nothing special about our choice to use the Hopf map \(\eta\). Indeed, the construction that lead to Proposition \ref{prop:pdnonconsum} can be extended to a broader context in which we have more general wedges of spheres, if we have the following three ingredients. 
    \begin{enumerate}
        \item[(i)] we have two wedges of spheres \(W_1\) and \(W_2\), in which all spheres are of dimension at least 2, and for each \(i=1,2\) there are maps \(g_i:S^{2n-1}\rightarrow W_i\) such that their homotopy cofibres \(M_i\) are \(2n\)-dimensional Poincar\'e Duality complexes
        \item[(ii)]  integers \(a,b>1\) such that \(2n>a+b\) and that spheres \(S^a\) and \(S^b\) are wedge summands of \(W_1\) and \(W_2\), respectively.
        \item[(iii)] there is some non-trivial class \(\alpha\in\pi_{2n-1}(S^{a+b-1})\).
    \end{enumerate} 
    This produces a Poincar\'e Duality complex with the homotopy type of \(M_1\#M_2\) after looping, but not before. In principle, one could go further and seat this section in a more general context by considering maps of a sphere into a wedge, say \(f:S^n\rightarrow A\vee B\); in effect we have been using our knowledge of the Hilton-Milnor Theorem to provide the necessary clarity when \(A\) and \(B\) are wedges of spheres.
\end{rem}

\section{An Application to the Vigu\'e-Poirrier Conjecture}\label{sec:vp}

We briefly recall the framework of the Vigu\'e-Poirrier Conjecture. Letting \(X\) be a simply connected space, we denote its \textit{free loop space} by \(\calL X\). A simply connected space \(X\) is called \textit{rationally elliptic} if \(dim(\pi_*(X)\otimes\Q)<\infty\), and called\textit{ rationally hyperbolic} otherwise \cite{fht}. Vigu\'e-Poirrier made the following conjecture in \cite{vp}.

\begin{conj}[Vigu\'e-Poirrier]
If \(X\) is rationally hyperbolic, then \(H_*(\calL X;\Q)\) grows exponentially.
\end{conj}

Recall also that, by a result of Lambrechts proved that the Conjectures holds for non-trivial connected sums of certain closed manifolds \cite{lambrechtsbetti}. Now we establish some terminology, which we take chiefly from \cites{ht21, fhtloop}. 

\begin{defn}
    A graded vector space \(V=\lbrace V_i\rbrace_{i\geq0}\) of finite type \textit{grows exponentially} if there exist constants \(1<C_1<C_2<\infty\) such that for some \(K\) \[C_1^k\leq \sum_{i\leq k}dim(V_i)\leq C_2^k\text{ for all }k\geq K.\] The \textit{log-index} of \(V\) is defined by \[\text{log-index}(V)=\limsup\limits_{i}\frac{\ln(dim(V_i))}{i}.\] For a simply connected space \(X\), let \(\text{log-index}(\pi_*(X))=\text{log-index}(\pi_{\geq2}(X)\otimes\Q)\). 
\end{defn}

Note that if \(X\) is rationally elliptic then \(\text{log-index}(\pi_*(X))=-\infty\) and if \(X\) is rationally hyperbolic then \(\text{log-index}(\pi_*(X))>0\). The following two stronger definitions were formulated by F\'elix, Haperin and Thomas in \cite{fhtloop}, after which they then give Theorem \ref{thm:fht} (see \cite{fhtloop}*{Theorem 3}).

\begin{defn} A graded vector space \(V\) as described above has \textit{controlled exponential growth} if \(\text{log-index}(V)\in(0,\infty)\) and for each \(\lambda>1\) there is an infinite sequence \(n_1<n_2<\cdots\) such that \(n_{i+1}<\lambda n_i\) for all \(i\geq0\) and \(dim(V_{n_i})=e^{\alpha_i n_i}\) with \(\alpha_i\rightarrow\text{log-index}(V)\).
\end{defn}

\begin{defn} \label{def:geg}
Let \(X\) be a simply connected topological space with rational homology of finite type, and such that \(\text{log-index}(H_*(\Omega X;\Q))\in(0,\infty)\). Then \(\calL X\) has \textit{good exponential growth} if \(H_*(\calL X;\Q)\) has controlled exponential growth and \[\text{log-index}(H_*(\calL X;\Q))=\text{log-index}(H_*(\Omega X;\Q)).\]
\end{defn}

\begin{thm}[F\'elix-Haperin-Thomas]\label{thm:fht}
Let \(F\rightarrow Y\rightarrow Z\) be a fibration between simply connected spaces with rational homology of finite type. If \[\text{log-index}(\pi_*(Z))<\text{log-index}(\pi_*(Y))\] then \(\calL Y\) has good exponential growth if and only if \(\calL F\) does. \hfill \qed
\end{thm}

We are now ready to provide our application using Theorem \ref{thm:PDconn}.

\begin{cor}\label{cor:vp}
Let $n>3$ and suppose that $M$, $N$ and $P$ are $n$-dimensional Poincar\'e Duality complexes that satisfy conditions (i)-(iv) of Theorem \ref{thm:PDconn}. If \[\text{log-index}(\pi_*(P))<\text{log-index}(\pi_*(N\#P))\] then \(\calL M\) has good exponential growth if and only if \(\calL(N\#P)\) does. 
\end{cor}

\begin{proof}
Recall from the proof of Theorem \ref{thm:PDconn} that we had homotopy cofibrations \[N_{n-1}\rightarrow M\xrightarrow{h_1} Q\text{\; and \;}N_{n-1}\rightarrow N\#P\xrightarrow{h_2} P\] in which both \(h_1\) and \(h_2\) have right homotopy inverses after looping. In particular, by Corollary \ref{1cor:splitfib}, these give rise to homotopy fibrations 
\begin{equation}\label{eqn:fibs}
\Omega P\ltimes N_{n-1}\rightarrow M\xrightarrow{h_1} Q\text{\; and \;}\Omega P\ltimes N_{n-1}\rightarrow N\#P\xrightarrow{h_2} P.
\end{equation}
Taking the second of these homotopy fibrations, since we supposed that \[\text{log-index}(\pi_*(P))<\text{log-index}(\pi_*(N\#P))\] we have by Theorem \ref{thm:fht} that \(\calL(N\#P)\) has good exponential growth if and only if \(\calL(\Omega P\ltimes N_{n-1})\) does. Since we have assumption \textit{(iii)} of Theorem \ref{thm:PDconn}, \(\Omega Q\simeq \Omega P\). It is evident that \[\text{log-index}(\pi_*(P))=\text{log-index}(\pi_*(Q))\] and by the result Theorem \ref{thm:PDconn} of we have \[\text{log-index}(\pi_*(M))=\text{log-index}(\pi_*(N\#P)).\] Therefore \(\text{log-index}(\pi_*(Q))<\text{log-index}(\pi_*(M))\) and we may once again apply Theorem \ref{thm:fht} to the first homotopy fibration of (\ref{eqn:fibs}) and deduce that the good exponential growth of \(\calL(\Omega P\ltimes N_{n-1})\) is equivalent to good exponential growth of \(\calL M\). 
\end{proof}

Note in particular that if \(P\) is rationally elliptic, the assumption of Corollary \ref{cor:vp} always holds. We record this in the Theorem below, and then close with an example.

\begin{thm}\label{thm:vp}
    Let $n>3$ and let $M$ be a $n$-dimensional Poincar\'e Duality complexes, and suppose that there exist \(n\)-dimensional Poincar\'e Duality complexes \(N\) and \(P\) such that \(\Omega M\simeq\Omega (N\#P)\), satisfying conditions (i)-(iv) of Theorem \ref{thm:PDconn}. If \(P\) is rationally elliptic, then \(\calL M\) has good exponential growth if and only if \(\calL(N\#P)\) does. \qed
\end{thm}

In the situation of Theorem \ref{thm:vp}, therefore, such an \(M\) satisfies the Vigu\'e-Poirrier Conjecture. Moreover, good exponential growth of \(\calL(N\#P)\) occurs in cases when \(N\) is rationally hyperbolic, or indeed when the cohomological condition of Lambrechts \cite{lambrechtsbetti} is satisfied. Such spaces also arise as pullback fibrations over connected sums (see for example \cite{chenfib}*{Theorem 4.2}). We close with a further example, using our construction of Section \ref{sec:conex}.

\begin{exa}
    Consider our Poincar\'e Duality complex \(\mathcal{M}_n\) from the main part of Section \ref{sec:conex}. We had the homotopy equivalence \[\Omega \mathcal{M}_n\simeq\Omega((S^n\times S^n)\#(S^{n-1}\times S^{n+1})).\] By \cite{fhtloop}*{Theorem 1}, \(\calL((S^n\times S^n)\#(S^{n-1}\times S^{n+1}))\) has good exponential growth. Moreover, though we only need one of them to be so, it is easily verified that both \(S^n\times S^n\) and \(S^{n-1}\times S^{n+1}\) are rationally elliptic. Thus Theorem \ref{thm:vp} applies, and \(\mathcal{M}_n\) satisfies the Vigu\'e-Poirrier Conjecture. 
\end{exa}

\bibliographystyle{amsplain}
\bibliography{bib}

\end{document}